\documentclass{amsart}
\usepackage{cite,color,array}

\newcommand{\dash}{\text{-}}
\newcommand{\s}{\mathrm{st}}

\advance\textheight by \topskip
\newtheorem{theorem}{Theorem}[section]
\newtheorem{corollary}[theorem]{Corollary}
\newtheorem{lemma}[theorem]{Lemma}

\numberwithin{equation}{section}

\title{Recurrence relations for patterns of type $(2,1)$ in flattened permutations}

\author{Toufik Mansour}
\address{Department of Mathematics, University of Haifa, 31905 Haifa, Israel}
\email{tmansour@univ.haifa.ac.il}

\author{Mark Shattuck}
\address{Department of Mathematics, University of Tennessee, Knoxville, TN 37996}
\email{shattuck@math.utk.edu}

\author{David G.L. Wang*}
\address{Department of Mathematics, University of Haifa, 31905 Haifa, Israel}
\email{david.combin@gmail.com, wgl@math.haifa.ac.il}



\begin{document}

\begin{abstract}
We consider the problem of counting the occurrences of patterns of the form $xy\dash z$ within flattened permutations of a given length. Using symmetric functions, we find recurrence relations satisfied by the distributions on $\mathcal{S}_n$ for the patterns 12-3, 21-3, 23-1 and 32-1, and develop a unified approach to obtain explicit formulas.
By these recurrences, we are able to determine simple closed form expressions for the number of permutations that, when flattened, avoid one of these patterns as well as expressions for the average number of occurrences.  In particular, we find that
the average number of 23-1 patterns and the average number of 32-1 patterns
in $\text{Flatten}(\pi)$, taken over all permutations $\pi$ of the same length, are equal, as are the number of permutations avoiding either of these patterns.
We also find that the average number of 21-3 patterns
in $\text{Flatten}(\pi)$ over all $\pi$ is the same as it is for 31-2 patterns.
\end{abstract}
\maketitle

\section{Introduction}

The pattern counting problem for permutations has been studied extensively from various perspectives in both enumerative and algebraic combinatorics; see, e.g., \cite{EN, Ki}.
The comparable problem has also been considered on other discrete structures such as $k$-ary words \cite{BM}, compositions \cite{MSi}, and set partitions \cite {MSY} (see also \cite{HM} and the references contained therein).

In his recent study~\cite{C} on finite set partitions,
Callan introduced the notion of flattened partitions.
In a previous paper~\cite{MSW},
we considered flattened permutations in the same sense
and obtained formulas for the generating functions which count the flattened permutations of size $n$ according to the number of
peaks and valleys.
Here,
we continue this work for some related statistics on flattened permutations.

Let $[n]=\{1,2,\ldots,n\}$ if $n \geq 1$, with $[0]=\emptyset$.  Denote
the set of permutations of~$[n]$ by~$\mathcal{S}_n$.
Let $\pi=\pi_1\pi_2\cdots\pi_n\in\mathcal{S}_n$.
A {\em pattern} is any permutation $\sigma$ of shorter length,
and an {\em occurrence} of~$\sigma$ in~$\pi$ is a subsequence of~$\pi$
that is order-isomorphic to~$\sigma$.
If $r$ denotes the number of occurrences of a given pattern $\sigma$ within a permutation $\pi$ in general, then the case that has been studied most often in previous research has been when $r=0$, i.e., the avoidance of $\sigma$ by $\pi$.
Relatively little work has been done concerning the case when $r>0$, and in what has been done, the
patterns were usually of length three.
Simple algebraic maps show that
the six patterns of length three are classified
into two classes with respect to the pattern counting enumeration,
that is, the class with the representative pattern
$\sigma=123$ (see \cite{NZ12,NZ96}) and
the class with the representative pattern $\sigma=132$ (see \cite{MV02} and references therein).

By specifying the length of adjacent letters allowed,
Claesson and Mansour~\cite{CM02} further generalized
the notion of patterns.
Precisely, a pattern $\sigma=\sigma_1\dash\sigma_2\dash\cdots\dash\sigma_k$
is said to be of {\em type} $(\ell_1,\ell_2,\ldots,\ell_k)$
if the subword~$\sigma_i$ of~$\sigma$ has length~$\ell_i$.
In this notation,
a classical pattern of length~$k$ is of type $(1,1,\ldots,1)$
which consists of~$k$ occurrences of~$1$.
In particular, the permutation~$\pi$ is said to contain a pattern $\tau=xy\dash z$ of type $(2,1)$
if there exist indices $2\le i<j\le n$ such that $\pi_{i-1}\pi_i\pi_j$
is order-isomorphic to $xyz$,
where $xyz$ is some permutation of $\{1,2,3\}$.
Otherwise, we say that $\pi$ avoids $\tau$.

Let $\pi\in\mathcal{S}_n$ be a permutation of length~$n$
represented in its {\em standard cycle form},
that is,
cycles arranged from left to right in ascending order
according to the size of the smallest elements,
where the smallest element is written first within each cycle.
Define $\text{Flatten}(\pi)$ to be the permutation of length $n$
obtained by erasing the parentheses
enclosing the cycles of $\pi$ and considering the resulting word.
For example, if $\pi=71564328 \in \mathcal{S}_8$,
then the standard cycle form of $\pi$ is $(172)(3546)(8)$ and $\text{Flatten}(\pi)=17235468$.

One can combine the ideas of the previous two paragraphs and say that a permutation $\pi$ contains a pattern $\tau$ in the \emph{flattened sense} if and only if $\text{Flatten}(\pi)$ contains $\tau$ in the usual sense and avoids $\tau$ otherwise.  Here, we will use this definition of pattern containment and consider the case when $\tau$ is a pattern of type $(2,1)$.  For example, the permutation $\pi=71564328$ avoids $23\dash1$ but has four occurrences of $31\dash 2$ by this definition since $\text{Flatten}(\pi)=17235468$ avoids $23\dash1$ but has four occurrences of $31\dash2$.

Let $\s$ denote the statistic on $\mathcal{S}_n$ which records the number of occurrences in the flattened sense of one of five patterns under consideration in this paper.
In accordance with the previous paper~\cite{MSW},
we will use the notation
\[
g^\s_n(a_1a_2\cdots a_k)
=\sum_{\pi}q^{\s(\text{Flatten}(\pi))},
\]
where $\pi$ ranges over all permutations of length~$n$ such that
$\text{Flatten}(\pi)$ starts with $a_1a_2\cdots a_k$.
It is easy to see that $g^\s_n(a_1a_2\cdots a_k)=0$ if $a_1\ne1$.
We will write $g_n=g^\s_n(1)$ when considering a particular pattern.

In this paper, we use symmetric functions to develop recurrences for the generating functions $g^\s_n(1)$ in the cases when $\s$ is the statistic recording the number of occurrences of $\tau$, where $\tau$ is any pattern of type $(2,1)$ (except for $13\dash2$).
As a consequence, we obtain simple closed formulas for the number of permutations avoiding a pattern of type $(2,1)$ in the flattened sense as well as the average number of occurrences of a pattern over all permutations of a given size.  We provide algebraic proofs of these results, as well as combinatorial proofs in all but two cases.  The results are summarized in Table \ref{tab1} below.

\begin{table}[htdp]
\caption{The number of avoiding and the average number of occurrences for patterns of type $(2,1)$ over flattened permutations of length $n$.}
\begin{center}
\begin{tabular}{|c|c|c|c|}
\hline
pattern  & number of avoiding  & average number & reference
\\ \hline\rule{0pt}{.8cm}\raisebox{1.4ex}[0pt]
{13-2}
&\raisebox{1.4ex}[0pt]{$\displaystyle2^{n-1}$}
&\raisebox{1.4ex}[0pt]{$\displaystyle{n^2+3n+8\over12}-H_n$}
&\raisebox{1.4ex}[0pt]{\cite{MW13X}}
\\ \hline\rule{0pt}{.9cm}\raisebox{1.7ex}[0pt]
{31-2}
&\raisebox{1.7ex}[0pt]{$\displaystyle{2n-2\choose n-1}$ }
&
&\raisebox{1.7ex}[0pt]{Corollary \ref{cor:31-2:average}}
\\
\cline{1-1}\cline{2-2}\cline{4-4}\rule{0pt}{1.1cm}\raisebox{2.4ex}[0pt]{21-3}
&\raisebox{2.7ex}[0pt]{$\displaystyle2\sum_{k=1}^{n-1}kS(n-1,k)$}
&\raisebox{5.3ex}[0pt]{$\displaystyle{n^3-3n^2+26n-12\over 12n}-H_n$}
&\raisebox{2.4ex}[0pt]{Corollary \ref{cor:21-3:average}}
\\ \hline\rule{0cm}{.5cm}\raisebox{.08cm}[0pt]
{32-1}
&
&
&\raisebox{.08cm}[0pt]{Corollary \ref{cor:32-1:average}}
\\ \cline{1-1} \cline{4-4}\rule{0pt}{.5cm}\raisebox{.08cm}[0pt]
{23-1}
&\raisebox{.4cm}[0pt]{$\displaystyle\sum_{k=1}^{n-1}2^kS(n-1,k)$}
&\raisebox{.38cm}[0pt]{$\displaystyle{n^2-9n-4\over 12}+H_n$}
&\raisebox{.08cm}[0pt]{Corollary \ref{cor:23-1:average}}
\\ \hline\rule{0pt}{1cm}\raisebox{1.4ex}[0pt]
{12-3}
&\raisebox{2.4ex}[0pt]{$\displaystyle-2\sum_{i=0}^{n-2}{n-2\choose i}(B_i+B_{i+1})\tilde{B}_{n-i-3}$}
&\raisebox{2ex}[0pt]{$\displaystyle{n^3+3n^2-40n+24\over 12n}+H_n$}
&\raisebox{2.2ex}[0pt]{Corollary \ref{cor:12-3:average}}
\\ \hline
\end{tabular}
\end{center}
\label{table}
\end{table}

Note that as an immediate consequence, we have the following result.

\begin{theorem}
For any pattern $p$ of type $(2,1)$,
the average number $\mathrm{avr}(n)$ of occurrences of $p$ in
$\text{Flatten}(\pi)$ over all permutations~$\pi$ of length~$n$
satisfies
\[
\lim_{n\to\infty}{\mathrm{avr}(n)\over n^2}={1\over 12}.
\]
\end{theorem}

We will need the following notation related to symmetric functions.
Let $X=\{x_1,x_2,\ldots,x_m\}$ be an ordered set.
Define
\begin{align*}
e_j(X)
&=\sum_{1\le i_1<i_2<\cdots< i_j\le m}x_{i_1}x_{i_2}\cdots x_{i_j},\\
e_j'(X)
&=\sum_{1\le i_1\le i_2-2\le i_3-4\le\cdots\le i_j-2(j-1)\le m}x_{i_1}x_{i_2}\cdots x_{i_j},\\
h_j(X)
&=\sum_{1\le i_1\le i_2\le\cdots\le i_j\le m}x_{i_1}x_{i_2}\cdots x_{i_j}.
\end{align*}
In other words,
$e_j(X)$ is the sum of products of any $j$ distinct elements of~$X$;
$e_j'(X)$ is the sum of products of any $j$ pairwise non-adjacent elements in~$X$;
$h_j(X)$ is the sum of products of any $j$ elements (non-distinct allowed) of~$X$.
For convenience, for any function $s_j(X)$ of these three,
let $s_0(X)=1$, $s_j(\emptyset)=\delta_{j,1}$ and $s_j(X)=0$ if $j<0$.

Throughout this paper,
we will make use of the Kronecker's delta notation $\delta_{i,j}$ defined by
\[
\delta_{i,j}=\begin{cases}
1,&\text{if $i=j$};\\
0,&\text{if $i\ne j$}.
\end{cases}
\]
We will follow the standard notation $[n]=\sum_{j=0}^{n-1}q^j={1-q^n\over 1-q}$ and
$[m]!=\prod_{i=1}^m[i]$, with
\[
{n\brack k}=\begin{cases}
{[n]!\over[k]![n-k]!},&\text{if }0\le k\le n;\\
0,&\text{otherwise},
\end{cases}•
\]
where $q$ is an indeterminate.  Note that $[n]\big|_{q=0}=1$ and $[n]\big|_{q=1}=n$.
Moreover, $[n]'|_{q=1}={n\choose 2}$.

Let $H_n=\sum_{k=1}^n{1\over k}$ denote the $n$-th harmonic number;
see Graham, Knuth and Patashnik~\cite{GKP94}.
Denote the Stirling number of the second kind by $S(n,k)$,
the $n$th Bell number by $B_n$, and
the $n$th complementary Bell number by $\tilde{B}_n$.

\section{Counting $31\dash2$-patterns}
For any $3\le i\le n$, we have
\begin{align}
g_n(1i)
&=\sum_{j<i}g_n(1ij)+\sum_{j>i}g_n(1ij)
=\sum_{j<i}q^{i-j-1}g_{n-1}(1j)+\sum_{j\ge i}g_{n-1}(1j)\notag\\
&=g_{n-1}+\sum_{j\le i-1}(q^{i-j-1}-1)g_{n-1}(1j).\label{rec1:31-2:g1i}
\end{align}

Define
\begin{align}
G_{n,r}(v)&=\sum_{i\ge2}g_{n,r}(1i)v^{i-2},\\
G_{r}(x,v)&=\sum_{n\ge 2}G_{n,r}(v)x^n=\sum_{n\ge 2}\sum_{i\ge2}g_{n,r}(1i)v^{i-2}x^{n}.
\end{align}
It follows that $G_{n,r}(1)=g_{n,r}$ and
$G_{r}(x,1)=\sum_{n\ge 2}g_{n,r}x^n$.
As before, we have $g_{n,r}(12)=2g_{n-1,r}$.

\begin{theorem}
For any integer $r\ge0$, we have
\begin{align}
\Bigl(1+{v^2x\over 1-v}\Bigr)G_{r}(x,v)
=x\sum_{j=0}^{r-1}v^{r-j+1}G_j(x,v)+{2-v\over 1-v}xG_{r}(x,1)+2(2+v)x^3\delta_{r,0}
+H_r(x,v),
\label{rec:31-2:Gr}
\end{align}
where
$H_r(x,v)=-x\sum_{s=0}^{r-1}\sum_{j\ge2}\sum_{n=3}^{j+r-s-1}g_{n,s}(1j)v^{j+r-s-1}x^n$.
\end{theorem}

\begin{proof}
Let $r\ge0$. Extracting the coefficient of $q^r$ in~(\ref{rec1:31-2:g1i}) gives
\begin{equation}\label{eq:2.1}
g_{n,r}(1i)-g_{n-1,r}+\sum_{j=2}^{i-2}g_{n-1,r}(1j)-\sum_{j=2}^{i-2}g_{n-1,r-i+j+1}(1j)=0,
\quad 3\le i\le n,
\end{equation}
with $g_{2,r}=g_{2,r}(12)=2\delta_{r,0}$.  Multiplying each of the four terms on the left-hand side of (\ref{eq:2.1}) by ${v^{i-2}x^{n}}$, and summing over $n \geq 3$ and $3\le i\le n$, yields
\begin{align*}
&\sum_{n\ge 3}\sum_{i=3}^{n}g_{n,r}(1i)v^{i-2}x^{n}
=G_{r}(x,v)-2xG_r(x,1)-4x^{3}\delta_{r,0},\\
&\sum_{n\ge 3}\sum_{i=3}^{n}g_{n-1,r}v^{i-2}x^{n}
={vx\over 1-v}G_{r}(x,1)
-{x\over 1-v}G_{r}(vx,1)
+2vx^{3}\delta_{r,0},\\
&\sum_{n\ge 3}\sum_{i=3}^{n}\sum_{j=2}^{i-2}g_{n-1,r}(1j)v^{i-2}x^{n}
={x\over 1-v}\Bigl(v^2G_r(x,v)-G_{r}(vx,1)\Bigr),
\end{align*}
and
\begin{align*}
&\sum_{n\ge 3}\sum_{i=3}^{n}\sum_{j=i-1-r}^{i-2}g_{n-1,r-i+j+1}(1j)v^{i-2}x^{n}\\
=&x\sum_{n\ge 3}\sum_{i=4}^{n+1}\sum_{j=i-1-r}^{i-2}g_{n,r-i+j+1}(1j)v^{i-2}x^{n}\\
=&x\sum_{i\ge 4}\sum_{j=i-1-r}^{i-2}\sum_{n\ge i-1}g_{n,r-i+j+1}(1j)v^{i-2}x^{n}\\
=&x\sum_{s\le r-1}\sum_{j\ge2}\sum_{n\ge r+j-s}g_{n,s}(1j)v^{r+j-s-1}x^{n},
\end{align*}
which combine to give \eqref{rec:31-2:Gr}.
\end{proof}

Taking $r=0$ in recurrence~(\ref{rec:31-2:Gr}) gives
\begin{align*}
\Bigl(1+{v^2x\over 1-v}\Bigr)G_{0}(x,v)={2-v\over 1-v}xG_0(x)+2(2+v)x^3.
\end{align*}
To solve this equation, we use the kernel method and substitute $v=C(x)=\frac{1-\sqrt{1-2x}}{2x}$ to obtain $$G_0(x)=\frac{x}{\sqrt{1-4x}}-x-2x^2$$
and
\begin{align}\label{eqAAG0}
G_{0}(x,C(x))=\lim_{v\rightarrow C(x)}G_0(x,v)=\frac{x^2(\sqrt{1-4x}+8x-1)}{1-4x}.
\end{align}

Taking $r=1$ in (\ref{rec:31-2:Gr}) and using the fact that $H_1(x,v)=-x\sum_{j\ge3}g_{j,0}(1j)v^jx^j=-\frac{2x^4v^3}{1-xv}$, we obtain
\begin{align}
\Bigl(1+{v^2x\over 1-v}\Bigr)G_1(x,v)
=xv^2G_0(x,v)+{2-v\over 1-v}xG_1(x)-2\frac{x^4v^3}{1-xv}.
\end{align}
Substituting $v=C(x)=\frac{1-\sqrt{1-2x}}{2x}$ into this equation, and using (\ref{eqAAG0}), yields
$$G_1(x)=\frac{(3x-1)(1-5x+2x^2)+(1-6x+7x^2)\sqrt{1-4x}}{x\sqrt{(1-4x)^3}},$$
and thus
\begin{align*}
G_{1}(x,C(x))&=\lim_{v\rightarrow C(x)}G_1(x,v)\\
&= \frac{(1-4x)(5x^4-x^3-16x^2+8x-1)-(17x^4+19x^3-30x^2+10x-1)\sqrt{1-4x}}{x\sqrt{(1-4x)^5}}.
\end{align*}
Continuing in this way for $r=2,3$, we obtain the following result.

\begin{corollary}
For $0\le r\le 3$, we have
\[
G_r^{31\dash2}(x)=\frac{a_r(x)+b_r(x)\sqrt{1-4x}}{\sqrt{(1-4x)^{2r+1}}},
\]
where
\begin{itemize}
\item[(i)]
$a_0(x)=x$, $b_0(x)=-x-2x^2$;
\item[(ii)]
$a_1(x)=\frac{1}{x}(3x-1)(1-5x+2x^2)$, $b_1(x)=\frac{1}{x}(1-6x+7x^2)$;
\item[(iii)]
$a_2(x)=\frac{1}{x}(1-12x+50x^2-76x^3+22x^4)$, $b_2(x)=\frac{1}{x}(-1+10x-32x^2+28x^3)$;
\item[(iv)]
$a_3(x)=\frac{1}{x^2}(2-37x+270x^2-972x^3+1748x^4-1346x^5+220x^6)$, $b_3(x)=\frac{1}{x^2}(-2+33x-208x^2+614x^3-824x^4+368x^5)$.
\end{itemize}
\end{corollary}

The second-order difference transformation of the above formula gives
\begin{equation}\label{rec:31-2:g1k}
g_n(1k)
=(q+1)g_n\bigl(1(k-1)\bigr)-q\cdotp g_n\bigl(1(k-2)\bigr)+(q-1)g_{n-1}\bigl(1(k-2)\bigr),
\quad 5\le k\le n.
\end{equation}•
We shall solve it with the initial values
\begin{equation}\label{ini:31-2:g1k}
\begin{split}
g_n(13)&=g_{n-1},\quad n\ge3,\\
g_n(14)&=g_{n-1}+2(q-1)g_{n-2},\quad n\ge4.
\end{split}
\end{equation}

\begin{theorem}
For any $n\ge2$, we have
\begin{equation}\label{rec:31-2:gn}
g_n=\sum_{j=1}^{\lfloor n/2\rfloor}(q-1)^{j-1}b_{n,j}g_{n-j},•
\end{equation}•
where
\[
b_{n,j}=\sum_{k=0}^{n-j-1}{n-k\over j}{n-j-1-k\choose j-1}{j-2+k\choose j-2}q^k.
\]
\end{theorem}

\begin{proof}
For any $k\ge3$ and any integer $j$, define $a_{k,j}$ by
$a_{3,j}=\delta_{j,1}$,
$a_{4,j}=\delta_{j,1}+2\delta_{j,2}$ and
\begin{equation}\label{def:31-2:a}
a_{k,j}=(q+1)a_{k-1,j}-q\cdotp a_{k-2,j}+a_{k-2,j-1},\quad k\ge 5.
\end{equation}
By~(\ref{rec:31-2:g1k}) and~(\ref{ini:31-2:g1k}), it is routine to verify that
\[
g_n(1k)=\sum_{j\le \lfloor k/2\rfloor}a_{k,j}(q-1)^{j-1}g_{n-j},
\quad 3\le k\le n.
\]
Consequently, for any $n\ge2$, we have
\[
g_n=\sum_{k\ge2}g_n(1k)
=ng_{n-1}+\sum_{j=2}^{\lfloor n/2\rfloor}(q-1)^{j-1}b_{n,j}g_{n-j},
\]
where $b_{n,j}=\sum_{k=3}^na_{k,j}$.
By~(\ref{def:31-2:a}), we have
\[
b_{n,j}
=(q+1)b_{n-1,\,j}
-qb_{n-2,\,j}
+b_{n-2,\,j-1}
+2\chi(j=2\text{ and }n\ge4),\quad j\ge2.
\]
It follows that
\begin{align*}
B(x,y)=\sum_{n\ge4}\sum_{j\ge2}b_{n,j}x^{n-4}y^{j-2}
={2-x\over(1-x)^2}\cdot{1\over(1-x)(1-qx)-x^2y}.
\end{align*}
We obtain the desired expression of $b_{n,j}$ by extracting the coefficient of
$x^{n-4}y^{j-2}$ from $B(x,y)$.
\end{proof}

\begin{corollary}\label{cor:31-2:average}
For any $n\ge1$,
the number of permutations~$\pi$ of length~$n$ with $\text{Flatten}(\pi)$ avoiding $31\dash2$
is ${2n-2\choose n-1}$, and
the average number of
occurrences of $31\dash2$ in $\text{Flatten}(\pi)$ over $\pi\in\mathcal{S}_n$ is given by
${n^3-3n^2+26n-12\over 12n}-H_n$.
\end{corollary}

\section{Recurrence in terms of symmetric functions}

\subsection{Counting $32\dash1$-patterns}

Let $3\le i\le n$. We have
\begin{align}
g_n(1i)
&=\sum_{j<i}g_n(1ij)+\sum_{j>i}g_n(1ij)
=\sum_{j<i}q^{j-2}g_{n-1}(1j)+\sum_{j\ge i}g_{n-1}(1j)\notag\\
&=g_{n-1}+\sum_{j\le i-1}(q^{j-2}-1)g_{n-1}(1j).\label{rec1:32-1:g1i}
\end{align}

Define
\begin{align}
G_{n,r}(v)&=\sum_{i}g_{n,r}(1i)v^{i-2},\\
G_{r}(x,v)&=\sum_{n\ge 3}G_{n,r}(v)x^n=\sum_{n\ge 3}\sum_{i}g_{n,r}(1i)v^{i-2}x^{n}.
\end{align}
It follows that $G_{n,r}(1)=g_{n,r}$ and
$G_{r}(x,1)=\sum_{n\ge 3}g_{n,r}x^n$.
As before, we have $g_{n,r}(1i)=2g_{n-1,r}$.

\begin{lemma}
For any integer $r\ge0$, we have
\begin{align}
\Bigl(1+{vx\over 1-v}\Bigr)G_{r}(x,v)
=&{x(2-v+2vx)\over 1-v}G_{r}(x,1)
-{2vx^2\over 1-v}G_r(vx,1)\notag\\
&+2x^3(2+v+2vx+2v^2x)\delta_{r,0}
+H_r(x,v),
\label{rec:32-1:Gr}
\end{align}
where
$H_r(x,v)={x\over 1-v}\sum_{n\ge3}\sum_{j=3}^ng_{n,r-j+2}(1j)x^{n}(v^{j-1}-v^{n})$.
\end{lemma}

\begin{lemma}
For any integer $r\ge0$, we have
\begin{align}
(1-v+vx)G_{r}(x,v)
=&x(2-v+2vx)G_{r}(x,1)
-2vx^2G_r(vx,1)\notag\\
&+2x^3(1-v)(2+v+2vx+2v^2x)\delta_{r,0}
+H_r(x,v),
\label{rec:32-1:Gr}
\end{align}
where
$H_r(x,v)=\sum_{n\ge3}\sum_{j=3}^ng_{n,r-j+2}(1j)x^{n+1}(v^{j-1}-v^{n})$.
\end{lemma}

\begin{proof}
Let $r\ge0$. Extracting the coefficient of $q^r$ in~(\ref{rec1:32-1:g1i}) gives
\begin{equation}\label{eq:32-1:gnr1i}
g_{n,r}(1i)-g_{n-1,r}+\sum_{j\le i-1}g_{n-1,r}(1j)-\sum_{j\le i-1}g_{n-1,r-j+2}(1j)=0,\quad 3\le i\le n.
\end{equation}
In the last sum in this formula, the subscript $j$ has upper bound
$\min(i-1,r+2)$.

Note that
$G_{2,r}(v)=g_{2,r}=2\delta_{r,0}$.  We obtain equation~(\ref{rec:32-1:Gr})
by multiplying~(\ref{eq:32-1:gnr1i}) by ${v^{i-2}x^{n}}$ and summing over $n\ge 3$ and $3\le i\le n$.
The expressions that result from performing these operations on the four summands in \eqref{eq:32-1:gnr1i} are
\begin{align*}
&\sum_{n\ge 3}\sum_{i=3}^{n}g_{n,r}(1i)v^{i-2}x^{n}
=G_{r}(x,v)-2xG_r(x,1)-4x^{3}\delta_{r,0},\\
&\sum_{n\ge 3}\sum_{i=3}^{n}g_{n-1,r}v^{i-2}x^{n}
={vx\over 1-v}G_{r}(x,1)
-{x\over 1-v}G_{r}(vx,1)
+2vx^{3}\delta_{r,0},\\
&\sum_{n\ge 3}\sum_{i=3}^{n}\sum_{j=2}^{i-1}g_{n-1,r}(1j)v^{i-2}x^{n}
={vx\over 1-v}G_r(x,v)
-{x\over 1-v}G_r(vx,1)
+2vx^3\delta_{r,0},
\end{align*}
and
\begin{align*}
&\sum_{n\ge 3}\sum_{i=3}^{n}\sum_{j=2}^{i-1}g_{n-1,r-j+2}(1j)v^{i-2}x^{n}\\
=&x\sum_{n\ge2}\sum_{i=3}^{n+1}\sum_{j=2}^{i-1}g_{n,r-j+2}(1j)v^{i-2}x^{n}
={x\over 1-v}\sum_{n\ge2}\sum_{j=2}^ng_{n,r-j+2}(1j)x^{n}(v^{j-1}-v^{n})\\
=&2vx^3(1+2x+2vx)\delta_{r,0}
+{2vx^2\over 1-v}\bigl(G_r(x,1)-G_r(vx,1)\bigr)
+H_r(x,v).
\end{align*}
This completes the proof.
\end{proof}

The first-order difference transformation applied to \eqref{rec1:32-1:g1i} gives the recurrence
\begin{equation}\label{rec:32-1:g1k}
g_n(1k)=g_n\bigl(1(k-1)\bigr)+(q^{k-3}-1)g_{n-1}\bigl(1(k-1)\bigr),\quad 4\le k\le n.•
\end{equation}•
We will solve it with the initial value $g_n(13)=g_{n-1}$.

\begin{lemma}\label{lemeee}
Let $k\geq3$. For all $j\geq1$,
$$e_j([1],\ldots,[k-3])=\frac{1}{(1-q)^j}\sum_{a=0}^{k-3}(-1)^aq^{\binom{a+1}{2}}\left[\begin{array}{c} k-3\\a\end{array}\right]_q\binom{k-3-a}{k-3-j}.$$
\end{lemma}
\begin{proof}
Let $F(x_1,\ldots,x_k)=\sum_{j=0}^ke_j(x_1,\ldots,x_k)z^j$.
By the definition of elementary symmetric functions,
we deduce
\begin{align*}
F([1],\ldots,[k-3])
&=\prod_{a=1}^{k-3}(1+[a]z)
=\biggl(1+\frac{z}{1-q}\biggr)^{k-3}\biggl(\frac{qz}{1-q+z};q\biggr)_{k-3}\\
&=\sum_{a=0}^{k-3}q^{\binom{a}{2}}\left[\begin{array}{c} k-3\\a\end{array}\right]_q\frac{q^a(-z)^a(1-q+z)^{k-3-a}}{(1-q)^{k-3}}\\
&=\sum_{a=0}^{k-3}\sum_{b=0}^{k-3-a}(-1)^aq^{\binom{a+1}{2}}\left[\begin{array}{c} k-3\\a\end{array}\right]_q\binom{k-3-a}{b}\frac{z^{a+b}}{(1-q)^{a+b}}.
\end{align*}
The desired formula now follows from comparing coefficients of $z^j$ on both sides of the above identity.
\end{proof}

\begin{theorem}
For any $n\ge2$, we have
\begin{align}\label{rec:32-1:gn}
g_n
&=ng_{n-1}+
\sum_{j=2}^{n-2}\Biggl(
\sum_{a=1}^{j}(-1)^{j-a}q^{{a\choose 2}}\sum_{k=0}^{n-2-j}{j-a+k\choose k}{j-1+k\brack a-1}
\Biggr)g_{n-j}.
\end{align}•
\end{theorem}•

\begin{proof}
For any $k\ge3$ and any integer $j$, define
$a_{k,j}=e_{j-1}([1],\,[2],\,\ldots,[k-3])$.
By~(\ref{rec:32-1:g1k}), it is routine to verify that
\[
g_n(1k)=\sum_{j=1}^{k-2}a_{k,j}(q-1)^{j-1}g_{n-j},\quad 3\le k\le n.
\]
Therefore,
\[
g_n=\sum_{k\ge2}g_n(1k)=ng_{n-1}+\sum_{j=2}^{n-2}b_{n,j}(q-1)^{j-1}g_{n-j},•
\]
where $b_{n,j}=\sum_{k=j+2}^ne_{j-1}([1],\,[2],\,\ldots,[k-3])$.
By Lemma \ref{lemeee}, we deduce
$$b_{n,j}=\frac{1}{(1-q)^j}\sum_{k=j+2}^n\sum_{a=0}^{k-3}(-1)^aq^{\binom{a+1}{2}}\left[\begin{array}{c} k-3\\a\end{array}\right]_q\binom{k-3-a}{k-2-j},$$
which gives \eqref{rec:32-1:gn}.
\end{proof}

\begin{corollary}\label{cor:32-1:average}
For any $n\ge2$, the number of permutations~$\pi$ of length~$n$ with $\text{Flatten}(\pi)$ avoiding
$32\dash1$ is
$\sum_{k=1}^{n-1}2^kS(n-1,k)$,
and
the average number of
occurrences of $32\dash1$ in $\text{Flatten}(\pi)$ over $\pi\in\mathcal{S}_n$ is given by
${n^2-9n-4\over 12}+H_n$.
\end{corollary}

\begin{proof} Let $n\ge2$.
Setting $q=0$ in~(\ref{rec:32-1:gn}), we get
\[
g_{n,0}=ng_{n-1,0}+\sum_{j=2}^{n-2}(-1)^{j-1}{n-2\choose j}g_{n-j,0}.
\]
Note that  $g_{1,0}=1$.
One can prove by induction that
\[
g_{n,0}={1\over e^2}\sum_{k\ge1}{2^kk^{n-1}\over k!}=\sum_{k=1}^{n-1}2^kS(n-1,k).
\]
The average number can be found in a similar manner as in the case $13\dash2$.
\end{proof}

For more information of the sequence $g_n(0)$, see the sequence~$A001861$ in OEIS~\cite{OEIS}.

\subsection{Counting $12\dash3$-patterns}

Let $3\le i\le n$. We have
\begin{align}
g_n(1i)
&=\sum_{j<i}g_n(1ij)+\sum_{j>i}g_n(1ij)
=\sum_{j<i}q^{j-i+1}g_{n-1}(1j)+q^{n-i}\sum_{j\ge i}g_{n-1}(1j)\notag\\
&=q^{n-i}g_{n-1}-\sum_{j\le i-1}(q^{n-i}-q^{j-i+1})g_{n-1}(1j).\label{rec1:12-3:g1i}
\end{align}
The first-order difference transformation of the above formula gives the recurrence
\begin{equation}\label{rec:12-3:g1k}
g_n(1k)
=q^{-1}g_n\bigl(1(k-1)\bigr)+(1-q^{n-k})g_{n-1}\bigl(1(k-1)\bigr),
\quad 4\le k\le n.
\end{equation}•
We will solve it with the initial value
\begin{equation}\label{ini:12-3:g1k}
g_n(13)=q^{n-3}g_{n-1}-2q^{n-3}(q^{n-3}-1)g_{n-2},\quad n\ge3.
\end{equation}

\begin{lemma}\label{lem:h}
Let $n\geq0$, $0\le j\le k-1$ and $X=\{[n+i]\colon0\le i\le k-j-1\}$. Then
$$h_{j-1}(X)
=\sum_{i=0}^{j-1}(-1)^i
{k-j-1+i\brack i}
\binom{k-2}{j-1-i}\frac{q^{in}}{(1-q)^{j-1}}.$$
\end{lemma}
\begin{proof}
By the definition of complete symmetric functions,
we have
$\sum_{j\ge0}h_{j}(X)z^{j}=\prod_{x\in X}(1-xz)^{-1}$ for any set~$X$.
Taking $X=\{[n+i]\colon0\le i\le k-j-1\}$, we deduce
\begin{align*}
\sum_{j\ge1}h_{j-1}(X)z^{j-1}
&=\biggl(1-\frac{z}{1-q}\biggr)^{j-k}
\prod_{i=0}^{k-j-1}\biggl(1+\frac{zq^{n}}{1-q-z}q^i\biggr)^{-1}\\
&=\sum_{a\geq0}(-1)^a
{k-j-1+a\brack a}\frac{q^{na}z^a}{(1-q)^a}\biggl(1-\frac{z}{1-q}\biggr)^{j-k-a}\\
&=\sum_{a,b\geq0}(-1)^a
{k-j-1+a\brack a}\binom{k-j+a+b-1}{b}\frac{q^{na}z^{a+b}}{(1-q)^{a+b}}.
\end{align*}
The desired formula now follows from comparing coefficients of $z^{j-1}$ on both sides of the above identity.
\end{proof}

\begin{theorem}
For any $n\ge3$, we have
\begin{equation}\label{rec:12-3:gn}
g_n=\sum_{j=2}^{n-1}c_{n,j}g_{n-j},•
\end{equation}•
where
\[
c_{n,j}=\sum_{i=0}^{j-1}\sum_{k=0}^{n-1-j}(-1)^i\left(
2{k+i\brack i}{k+j-1\choose j-i-1}
-{k+i-1\brack i}{k+j-2\choose j-i-1}\right)q^{(i+1)(n-j-k-1)}.
\]
\end{theorem}

\begin{proof}
By~(\ref{rec:12-3:g1k}) and~(\ref{ini:12-3:g1k}),
it is routine to verify that
\[
{g_n(1k)}=q^{n-k}\sum_{j=1}^{k-1}(1-q)^{j-1}a_{n,k,j}g_{n-j}, \qquad 3\le k\le n,
\]
where $a_{n,k,j}
=2h_{j-1}\bigl(\bigl\{[n-i]\colon j+1\le i\le k\bigr\}\bigr)
-h_{j-1}\bigl(\bigl\{[n-i]\colon j+2\le i\le k\bigr\}\bigr)$.
Therefore,
\begin{equation}\label{rec:12-3:gna}
g_n=\sum_{k\ge2}g_n(1k)
=\bigl(2q^{n-2}+[n-2]\bigr)g_{n-1}+\sum_{j=2}^{n-1}b_{n,j}(1-q)^{j-1}g_{n-j},•
\end{equation}•
where
\begin{equation}\label{def:12-3:a}
b_{n,j}
=\sum_{k=j+1}^n
\biggl(2h_{j-1}\bigl(\bigl\{[n-i]\colon j+1\le i\le k\bigr\}\bigr)
-h_{j-1}\bigl(\bigl\{[n-i]\colon j+2\le i\le k\bigr\}\bigr)\biggr)q^{n-k}.
\end{equation}
Recurrence \eqref{rec:12-3:gn} now follows from Lemma \ref{lem:h}.
\end{proof}

\begin{corollary}\label{cor:12-3:average}
For any $n\ge2$, the number of permutations~$\pi$ of length~$n$ with $\text{Flatten}(\pi)$ avoiding $12\dash3$ is $-2\sum_{i=0}^{n-2}{n-2\choose i}(B_i+B_{i+1})\tilde{B}_{n-i-3}$, and
the average number of
occurrences of $12\dash3$ in $\text{Flatten}(\pi)$ over $\pi\in\mathcal{S}_n$ is given by ${n^3+3n^2-40n+24\over 12n}+H_n$.
\end{corollary}

\begin{proof}
Letting $q=0$ in~(\ref{rec:12-3:gn}), we get
\[
g_n(0)=\sum_{j=1}^{n-2}\Biggl({n-3\choose j-1}+{n-4\choose j-2}\Biggr)g_{n-j}(0),
\quad n\ge4,
\]
with $g_1(0)=1$ and $g_2(0)=g_3(0)=2$. Define $G(x)=\sum_{n\ge2}g_n(0){x^{n-2}\over (n-2)!}$.
Then the above recurrence translates to
\[
G''(x)=(e^xG(x))'+e^xG(x).
\]
Solving this differential equation gives
\[
G(x)=2(e^x+1)e^{e^x-1}\biggl(1-\int_{0}^xe^{1-e^t} dt\biggr)-2.
\]
Note that
\begin{align*}
e^{e^x-1}&=\sum_{n\ge0}B_n{x^n\over n!},\\
e^{e^x-1+x}&=\sum_{n\ge0}B_{n+1}{x^n\over n!},\\
\int_{0}^xe^{1-e^t} dt&=\int_0^x\sum_{n\ge0}\tilde{B}_n{t^n\over n!} dt
=\sum_{n\ge1}\tilde{B}_{n-1}{x^{n}\over n!},
\end{align*}
Recall that the sequence $\{\tilde{B}_n\}$ contains both positive and negative integers. (See Rao Uppuluri and Carpenter \cite{RC69} and entry A000587 in OEIS \cite{OEIS}).
We define $\tilde{B}_{-1}=-1$. Then
extracting the coefficient of $x^{n-2}$ from the formula above for $G(x)$ gives
\[
g_n(0)=-2\sum_{i=0}^{n-2}{n-2\choose i}(B_i+B_{i+1})\tilde{B}_{n-i-3},\quad n\ge3,
\]
which completes the proof of the first statement.

For the average number of occurrences, we
differentiate both sides of (\ref{rec:12-3:gn}) and set $q=1$ to obtain
\begin{align*}
g_n'(1)
&=\Bigl(2(n-2)q^{n-3}+[n-2]'\Bigr)\Big|_{q=1}g_{n-1}(1)+ng_{n-1}'(1)-b_{n,2}\big|_{q=1}g_{n-2}(1)\\
&=\biggl(2(n-2)+{n-2\choose 2}\biggr)(n-1)!+ng_{n-1}'(1)-{(n+2)(n-2)(n-3)\over3}(n-2)!.
\end{align*}
The desired result now follows from solving this recurrence and noting that the average number of occurrences is given by $g_n'(1)/n!$.
\end{proof}

\section{Recurrence in terms of generalized symmetric functions}
\subsection{Counting $23\dash1$-patterns}

Let $3\le i\le n$. We have
\begin{align}
g_n(1i)
&=\sum_{j<i}g_n(1ij)+\sum_{j>i}g_n(1ij)
=\sum_{j<i}g_{n-1}(1j)+\sum_{j\ge i}q^{i-2}g_{n-1}(1j)\notag\\
&=q^{i-2}g_{n-1}+(1-q^{i-2})\sum_{j\le i-1}g_{n-1}(1j).\label{rec1:23-1:g1i}
\end{align}
The first-order difference transformation of the above formula gives the recurrence
\begin{equation}\label{rec:23-1:g1k}
g_n(1k)
=-{q^{k-3}g_{n-1}\over[k-3]}
+{[k-2]\over[k-3]}g_n\bigl(1(k-1)\bigr)
+(1-q)[k-2]g_{n-1}\bigl(1(k-1)\bigr),\quad 4\le k\le n.
\end{equation}
We will solve it with the initial value
$g_n(13)=q\cdotp g_{n-1}+2(1-q)g_{n-2}$.

\begin{theorem}
For all $n\geq2$, we have
\begin{equation}\label{rec:23-1:gn}
g_n=\bigl(1+[n-1]\bigr)g_{n-1}+\sum_{j=2}^{n-1}b_{n,j}(1-q)^{j-1}g_{n-j},•
\end{equation}•
where
\[
b_{n,2}
=\sum_{k=1}^{n-2}[k]\bigl(1+[k]\bigr)
={(2-q)n\over(q-1)^2}
+{q^3-3q^2+q+4\over(q-1)^3(q+1)}
+{(q-3)q^{n-1}\over(q-1)^3}
+{q^{2n-2}\over(q-1)^3(q+1)},
\]
and for $j\ge3$,
\[
b_{n,j}=\sum_{k=j+1}^n
\sum_{1\le i_1<i_2<\cdots<i_{j-2}\le k-3}\bigl(1+[i_1]\bigr)[i_1][i_2]\cdots[i_{j-2}][k-2].
\]
\end{theorem}

\begin{proof}
For any $3\le k\le n$ and any integer $j$, define $d_{n,k,j}$ by
$d_{n,3,j}=q\delta_{j,1}+2(1-q)\delta_{j,2}$ and
\[
d_{n,k,j}
=-\frac{q^{k-3}}{[k-3]}\delta_{j,1}
+\frac{[k-2]}{[k-3]}d_{n,k-1,j}
+(1-q)[k-2]d_{n-1,k-1,j-1},
\quad 4\le k\le n.
\]
By~(\ref{rec:23-1:g1k}), it is easy to verify that $g_n(1k)=\sum_{j=1}^{k-1}d_{n,k,j}g_{n-j}$
for any $3\le k\le n$. On the other hand,
we can solve $d_{n,k,j}$ by iteration as follows. For any $k\ge4$, we have
\begin{align*}
d_{n,k,1}&=-[k-2]\sum_{j=3}^{k-1}{q^{k-j}\over[k-j][k-j+1]}+[k-2]d_{n,3,1}=q^{k-2},\\
d_{n,k,2}&=[k-2]d_{n,3,2}+(1-q)[k-2]\sum_{j=1}^{k-3}q^j
=(1-q)[k-2]\bigl(1+[k-2]\bigr),
\end{align*}
and for $j\ge3$,
\begin{align*}
d_{n,k,j}
&=(1-q)[k-2]\sum_{i=j}^{k-1}d_{n-1,i,j-1}\\
&=(1-q)^{j-2}[k-2]
\sum_{2<i_{j-2}<\cdots<i_2<i_1<k}[i_1-2][i_2-2]\cdots[i_{j-3}-2]d_{n-j+2,i_{j-2},2}\\
&=(1-q)^{j-1}[k-2]a_{k,j},
\end{align*}
where
\[
a_{k,j}
=\sum_{1\le i_1<i_2<\cdots<i_{j-2}\le k-3}\bigl(1+[i_1]\bigr)[i_1][i_2]\cdots[i_{j-2}],\quad j\ge3.
\]
Therefore,
for all $3\le k\le n$, we have
\[
g_n(1k)=q^{k-2}g_{n-1}+[k-2]\sum_{j=2}^{k-1}a_{k,j}(1-q)^{j-1}g_{n-j},
\]
where $a_{k,2}=1+[k-2]$.
Consequently,
we obtain the desired formula by using $g_n=\sum_{k=2}^ng_n(1k)$.
\end{proof}

\begin{corollary}\label{cor:23-1:average}
For any $n\ge2$,
the number of permutations $\pi$ of length~$n$ with $\text{Flatten}(\pi)$ avoiding $23\dash1$
is $\sum_{k=1}^{n-1}2^kS(n-1,k)$, and
the average number of
occurrences of $23\dash1$ in $\text{Flatten}(\pi)$ over $\pi\in\mathcal{S}_n$ is given by
${n^2-9n-4\over 12}+H_n$.
\end{corollary}

\begin{proof}
Let $n\ge2$. Taking $q=0$ in the recurrence~(\ref{rec:23-1:gn}), we obtain
\begin{equation}\label{rec:23-1:avoidance}
g_n(0)=2\sum_{j=1}^{n-1}{n-2\choose j-1}g_{n-j}(0).
\end{equation}•
Note that $g_1(0)=1$.
One can prove by induction that
\[
g_n(0)={1\over e^2}\sum_{k\ge1}{2^kk^{n-1}\over k!}=\sum_{k=1}^{n-1}2^kS(n-1,k).
\]
The average number of occurrences may be obtained as it was for previous patterns.
\end{proof}

\subsection{Counting $21\dash3$-patterns}

Let $3\le i\le n$. We have
\begin{align}
g_n(1i)
&=\sum_{j<i}g_n(1ij)+\sum_{j>i}g_n(1ij)
=\sum_{j<i}q^{n-i}g_{n-1}(1j)+\sum_{j\ge i}g_{n-1}(1j)\notag\\
&=g_{n-1}-(1-q^{n-i})\sum_{j\le i-1}g_{n-1}(1j).\label{eq:21-3:g1i}
\end{align}
In particular, since $g_n(12)=2g_{n-1}$, we have
\begin{equation}\label{fm:21-3:g13}
g_n(13)=g_{n-1}-(1-q^{n-3})g_{n-1}(12)
=g_{n-1}+2(q-1)[n-3]g_{n-2},\quad n\ge3.
\end{equation}
So we can focus on $n\ge4$.
The first-order difference transformation of~(\ref{eq:21-3:g1i}) gives
\begin{align}\label{rec:21-3:g1k}
g_n(1k)
={q^{n-k}g_{n-1}\over[n-k+1]}
+{[n-k]g_n\bigl(1(k-1)\bigr)\over[n-k+1]}
+(q-1)[n-k]g_{n-1}\bigl(1(k-1)\bigr),\quad 4\le k\le n.
\end{align}

\begin{theorem}
For all $n\geq2$, we have
\begin{equation}\label{rec:21-3:gn}
g_n=ng_{n-1}+\sum_{j=2}^{n-1}b_{n,j}(q-1)^{j-1}g_{n-j},
•\end{equation}•
where
\[
b_{n,2}
=\sum_{k=3}^n(k-1)[n-k]
=-{n^2\over 2(q-1)}+{(q-3)n\over2(q-1)^2}+{q(2q^{n-2}-q^{n-3}+q-2)\over(q-1)^3},
\]
and for $j\ge3$,
\[
b_{n,j}=\sum_{k=j+1}^n
\sum_{n-k\le i_1\le i_2\le \cdots\le i_{j-2}\le n-j-1}(n-j-i_{j-2}+1)[i_1][i_2]\cdots[i_{j-2}][n-k].
\]
\end{theorem}

\begin{proof}
For any $3\le k\le n$ and any integer $j$, define $d_{n,k,j}$ by
$d_{n,3,j}=\delta_{j,1}+2(q-1)[n-3]\delta_{j,2}$ and
\[
d_{n,k,j}
=\frac{q^{n-k}}{[n-k+1]}\delta_{j,0}
+\frac{[n-k]}{[n-k+1]}d_{n,k-1,j}
+(q-1)[n-k]d_{n-1,k-1,j-1},\quad 4\le k\le n.
\]
By~(\ref{rec:21-3:g1k}), it is easy to verify that $g_n(1k)=\sum_{j=1}^{k-1}d_{n,k,j}g_{n-j}$
for any $3\le k\le n$. On the other hand,
we can solve $d_{n,k,j}$ by iteration as follows:
\begin{align*}
d_{n,k,1}&=\frac{[n-k]}{[n-3]}+\sum_{j=4}^k\frac{[n-k]}{[n-j][n+1-j]}q^{n-j}=1,\\
d_{n,k,2}&=\frac{[n-k]}{[n-3]}d_{n,3,2}+(k-3)(q-1)[n-k]=(q-1)(k-1)[n-k],
\end{align*}
and for $j\ge3$,
\begin{align*}
d_{n,k,j}&=(q-1)[n-k]\sum_{i=j+1}^{k-1}d_{n-1,i,j-1}\\
&=(q-1)^{j-1}[n-k]\sum_{3\le i_{j-2}<\cdots<i_1\le k-1}(i_{j-2}-1)
\prod_{\ell=1}^{j-2}[n-\ell-i_\ell].
\end{align*}
Combining these formulas, we obtain
$$g_n(1k)=g_{n-1}+[n-k]\sum_{j=2}^{k-1}a_{n,k,j}(q-1)^{j-1}g_{n-j},\quad 3\le k\le n,$$
where $a_{n,k,2}=k-1$ and
\begin{equation}\label{def:21-3:a}
a_{n,k,j}
=\sum_{n-k\le i_1\le i_2\le \cdots\le i_{j-2}\le n-j-1}(n-j-i_{j-2}+1)[i_1][i_2]\cdots[i_{j-2}].
\end{equation}
We derive~(\ref{rec:21-3:gn}) by using $g_n=\sum_{k=2}^ng_n(1k)$.
This completes the proof.
\end{proof}

\begin{corollary}\label{cor:21-3:average}
For any $n\ge2$,
the number of permutations~$\pi$ of length~$n$ with $\text{Flatten}(\pi)$ avoiding $21\dash3$ equals
$2\sum_{k=1}^{n-1}kS(n-1,k)$,
and
the average number of
occurrences of $21\dash3$ in $\text{Flatten}(\pi)$ over $\pi\in\mathcal{S}_n$ is given by
${n^3-3n^2+26n-12\over 12n}-H_n$.
\end{corollary}

\begin{proof}
Let $n\ge3$. Letting $q=0$ in~(\ref{rec:21-3:gn}) gives
\begin{equation}\label{rec:21-3:avoidance}
g_n(0)=ng_{n-1}(0)-{n(n-3)\over 2}g_{n-2}(0)
+\sum_{j=3}^{n-1}(-1)^{j-1}\Biggl({n-2\choose j}+{n-3\choose j-1}\Biggr)g_{n-j}(0).
\end{equation}
Note that $g_1(0)=1$ and $g_2(0)=2$.
By a routine application of the generating function technique,
we may deduce that
\[
\sum_{n\ge0}g_{n+2}(0){x^n\over n!}=2e^{e^x+2x-1},
\]
which gives the desired formula of $g_n(0)$.
Differentiating both sides of (\ref{rec:21-3:gn}), and setting $q=1$, yields
\[
g_n'(1)
=ng_{n-1}'(1)+b_{n,2}\big|_{q=1}g_{n-2}(1)
=ng_{n-1}'(1)+{(n+2)(n-2)(n-3)\over6}(n-2)!.
\]
Solving this recurrence, we obtain the requested formula for the average number $g_n'(1)/n!$.
\end{proof}

\section{Combinatorial proofs}

In this section, we provide combinatorial proofs of Corollaries \ref{cor:31-2:average}, \ref{cor:32-1:average}, and \ref{cor:23-1:average} and of the statements concerning the average number of occurrences in Corollaries \ref{cor:12-3:average} and \ref{cor:21-3:average}.  We first prove the statements concerning the avoidance of the pattern in question.\\

\noindent\textbf{Combinatorial proofs of Corollaries \ref{cor:32-1:average} and \ref{cor:23-1:average} (avoiding).}\\

We first treat the case $23\dash1$.  Recall that the Stirling number of the second kind $S(m,k)$ counts the partitions of an $m$-element set into exactly $k$ blocks.  Then the sum $\sum_{k=1}^{n-1}2^kS(n-1,k)$ counts the partitions of the set $\{2,3,\ldots,n\}$ having any number of blocks in which some subset of the blocks are marked.  We will denote the set of such partitions by $\pi(n)^*$.  Let $\pi=B_1/B_2/ \cdots/B_k \in \pi(n)^*$, $1 \leq k \leq n-1$, where the $B_i$ are arranged in \emph{ascending} order of smallest elements and some subset of the $B_i$ are marked.  Furthermore, we assume within each block $B_i$ that the elements are written in \emph{descending} order.  Finally, let $m_i$ denote the smallest element of block $B_i$, $1 \leq i \leq k$.

We now transform $\pi$ into a permutation of size $n$ whose flattened form avoids the pattern $23\dash1$. We start by writing the element $1$ in a cycle by itself as $(1\cdots)$. We first consider the block $B_1$.  If $B_1$ is not marked, then write the elements of $B_1$ in descending order after $1$ within its cycle to obtain the longer cycle $(1B_1\cdots)$.  If $B_1$ is marked, then write all elements of $B_1$ except for the last one in the cycle with $1$ and start a new cycle with $m_1=2$; at this point, one would have two cycles $(1\widetilde{B}_1\cdots)$ and $(2\cdots)$, where $\widetilde{B}_1=B_1-\{m_1\}$.

Continue in this fashion, inductively, as follows.  If $i \geq 2$ and block $B_i$ is not marked, then write all of the elements of $B_i$ at the end of the last current cycle, while if $B_i$ is marked, write all of the elements of $B_i$ except $m_i$ at the end of the last current cycle and then write the element $m_i$ in a cycle by itself.  Doing this for each of the $k$ blocks of $\pi$ yields a permutation $\sigma$ of length $n$ which avoids $23\dash1$ such that $\text{Flatten}(\sigma)$ has exactly $k$ ascents.  The above procedure is seen to be reversible, and hence is a bijection, upon considering whether or not the smaller number in an ascent within $\text{Flatten}(\sigma)$ starts a new cycle of $\sigma$.

For example, if $n=10$ and $\pi=\{6,5,2\},\{10,7,3\},\{4\},\{9,8\} \in \pi(10)^*$, with the second and third blocks marked, then the corresponding permutation in standard cycle form would be $\sigma=(1,6,5,2,10,7),~(3),~(4,9,8)$.

For Corollary \ref{cor:32-1:average}, we now define a bijection between permutations avoiding $23\dash1$ and those avoiding $32\dash1$ in the flattened sense.  To do so, given $\sigma$ avoiding $23\dash1$, let $t_1=1<t_2<\cdots<t_\ell$ denote the set of numbers consisting of the first (i.e., left) letters of the ascents of $\text{Flatten}(\sigma)$, going from left to right. Note that the $t_i$ are increasing since there is no occurrence of $23\dash1$ in $\text{Flatten}(\sigma)$.  Given $1 < i \leq \ell$, let $\alpha_i$ denote the sequence (possible empty) of numbers occurring between $t_{i-1}$ and $t_i$ in $\text{Flatten}(\sigma)$.  Note that the letters of an $\alpha_i$ must belong to the same cycle of $\sigma$, by definition of the $t_i$, since $\sigma$ is assumed to be in standard cycle form.  Write the sequence $\alpha_i$ in reverse order for each $i$, where the letters remain in the same cycle of $\sigma$.  If $\sigma'$ denotes the resulting permutation, then it may be v
 erified that the mapping $\sigma \mapsto \sigma'$ is the requested bijection.  For example, if $\sigma$ is as above, then the order of the letters between $1$ and $2$ and between $2$ and $3$ is reversed, which gives $\sigma'=(1,5,6,2,7,10),(3),(4,9,8)$.\linebreak

\indent Given a permutation $\rho=\rho_1\rho_2\cdots\rho_n$ which avoids $32\dash1$, the above mapping is reversed by considering the subsequence $\rho_{i_r}$ of $\text{Flatten}(\rho)$ where $\rho_{i_1}=1$ and $\rho_{i_r}$ is the smallest letter to the right of $\rho_{i_{r-1}}$ if $r>1$ and changing the order of the letters between $\rho_{i_{r-1}}$ and $\rho_{i_r}$ for each $r$. \hfill \qed\\

\noindent\textbf{Combinatorial proof of Corollary \ref{cor:31-2:average} (avoiding).}\\

We will show that a permutation avoids $31\dash2$ in the flattened sense if and only if it avoids $3\dash1\dash2$, whence the result will follow from Theorem 2.4 in \cite{MS} where a combinatorial proof was given.  Clearly, a permutation avoiding $3\dash1\dash2$ also avoids $31\dash2$.  So suppose a permutation $\sigma$ contains an occurrence of $3\dash1\dash2$ in the flattened sense.  We will show that it must contain an occurrence of $31\dash2$.  Let $\text{Flatten}(\sigma)=\sigma_1\sigma_2\cdots \sigma_n$, which we'll denote by $\sigma'$.

First suppose that there is at least one ascent to the right of $n$ in $\sigma'$.  Let $j$ denote the index of the left-most such ascent.  That is, there exist indices $i$ and $j$ with $i<j<n$ such that $\sigma_i=n$, $\sigma_i>\sigma_{i+1}>\cdots>\sigma_j$ and $\sigma_{j+1}>\sigma_j$.  Since $\sigma_j<\sigma_{j+1}<\sigma_i$, there exists an index $\ell$ with $i<\ell\leq j$ such that $\sigma_\ell<\sigma_{j+1}<\sigma_{\ell-1}$.  Then the subsequence $\sigma_{\ell-1}\sigma_{\ell}\sigma_{j+1}$ would be an occurrence of $31\dash2$ in $\sigma'$, which completes this case.

On the other hand, suppose there is no ascent in $\sigma'$ to the right of the letter $n$.  If $\sigma_i=n$, then $\sigma_i>\sigma_{i+1}>\cdots>\sigma_n$.  Apply now the reasoning of the previous paragraph on the subpermutation $\sigma_1\sigma_2\cdots \sigma_{i-1}$, considering instead of $n$, the largest element among the first $i-1$ positions of $\sigma'$.  If an occurrence of $31\dash2$ arises as before, then we are done.  Otherwise, continue with still a smaller subpermutation.  If no occurrence of $31\dash2$ arises before all of the positions of $\sigma'$ are exhausted, then it must be the case that there is the following decomposition of $\sigma'$:
$$\sigma'=T_rT_{r-1}\cdots T_1,$$
for some $r \geq 1$, where $T_1$ starts with the letter $n$ and is decreasing and $T_i$, $1 < i \leq r$, starts with the largest letter to the left of $T_{i-1}T_{i-2}\cdots T_1$ followed by a possibly empty decreasing sequence.

Now suppose $\sigma'$ contains an occurrence of $3\dash1\dash2$ consisting of the letters $x$, $y$, and $z$, respectively.  If $m_k=\max(T_k)$, $1 \leq k \leq r$, then $m_1>m_2>\cdots>m_r$, which implies that we may assume that $x$ and $y$ belong to the same block $T_i$ for some $i$, with $x=m_i$, and that $z$ belongs to a block $T_j$ for some $j$ with $j<i$.  (Note that if $x\in T_{i_1}$ and $y \in T_{i_2}$, where $i_1>i_2$, then we have $m_{i_2}>m_{i_1}\geq x>y$, with $m_{i_2}$ occurring to the left of $y$ in $\sigma'$.)  Since the letters are decreasing between $x$ and $y$, inclusive, with $x>z>y$, there must exist an occurrence of $31\dash2$ in $\sigma'$ where the $3$ and $1$ correspond to a pair of adjacent letters between $x$ and $y$ (and possibly including $x$ or $y$) and the $2$ corresponds to the letter $z$.  Thus, $\sigma'$ contains an occurrence of $31\dash2$ in all cases, which completes the proof.  \hfill \qed\\

Given a pattern $\tau$ of type $(2,1)$, we will refer to an occurrence of $\tau$ (in the flattened sense) in which the $2$ corresponds to the actual letter $i$ as an $i$-occurrence of $\tau$ and an occurrence in which the $2$ and $3$ correspond to the letters $i$ and $j$, respectively, as an $(i,j)$-occurrence.  In the proofs that follow, $\text{tot}(\tau)$ denotes the total number of occurrences of the pattern $\tau$ under consideration.  Furthermore, given positive integers $m$ and $n$, we let $[m,n]=\{m,m+1,\ldots,n\}$ if $n\geq m$, with $[m,n]=\emptyset$ if $m>n$.\\

\noindent\textbf{Combinatorial proofs of Corollaries \ref{cor:32-1:average} and \ref{cor:23-1:average} (average).}\\

We first treat the case $32\dash1$ and argue that the total number of $i$-occurrences of $32\dash1$ (in the flattened sense) within all of the permutations of size $n$ is given by $(n-i)\binom{i-1}{2}\frac{(n-1)!}{i}$ for $i \in [3,n-1]$.  Summing over $i$ would then give the total number of occurrences of $32\dash1$.

Note that within an $i$-occurrence of $32\dash1$, the letter $i$ cannot start a cycle since there is a letter to the right of it in the flattened form which is smaller.  Note further that the position of $j$ is determined by that of $i$'s within an $(i,j)$-occurrence of $32\dash1$, where $i+1\leq j \leq n$.  Given $i$ and $j$, we count the permutations of size $n$ for which there are exactly $r$ $(i,j)$-occurrences of $32\dash1$, respectively, where $1 \leq r \leq i-2$.  Note that the position of $i$ is determined within such permutations once the positions of the elements of $[i-1]$ have been, which also determines the position of $j$ (note that $i$ must be placed within a current cycle so that there are exactly $r$ members of $[i-1]$ to its right within the flattened form).

Thus, there are $\frac{(n-1)!}{i}$ such permutations for each $r$, which implies that the total number of $(i,j)$-occurrences of $32\dash1$ is given by
$$\sum_{r=1}^{i-2}r\frac{(n-1)!}{i}=\binom{i-1}{2}\frac{(n-1)!}{i}.$$
Since there are $n-i$ choices for $j$, given $i$, with each choice yielding the same number of $(i,j)$-occurrences of $32\dash1$, it follows that there are $(n-i)\binom{i-1}{2}\frac{(n-1)!}{i}$ $i$-occurrences of $32\dash1$, as desired.  Summing over $3 \leq i \leq n-1$, and simplifying, then gives
\begin{align*}
\text{tot}(32\dash1)&=\sum_{i=3}^{n-1}(n-i)\binom{i-1}{2}\frac{(n-1)!}{i}=n!\sum_{i=3}^{n-1}\left(\frac{i-3}{2}+\frac{1}{i}\right)-(n-1)!\binom{n-1}{3}\\
&=\frac{n!}{2}\binom{n-3}{2}+n!\left(H_{n-1}-\frac{3}{2}\right)-(n-1)!\binom{n-1}{3}\\
&=\frac{n!}{12}(n^2-9n-4)+n!H_n.
\end{align*}
Dividing by $n!$ yields the average value formula found in Corollary \ref{cor:32-1:average}.

Writing the letter corresponding to $3$ directly after (instead of directly before) the letter corresponding to $2$ shows that the total number of $(i,j)$-occurrences of $23\dash1$ is the same as the total number of $(i,j)$-occurrences of $32\dash1$ for each $i$ and $j$, which implies Corollary \ref{cor:23-1:average}. \hfill \qed\\

\noindent\textbf{Combinatorial proof of Corollary \ref{cor:31-2:average} (average).}\\

Similar reasoning as in the prior proof shows that the total number of $i$-occurrences of $31\dash2$ in the flattened sense within all of the permutations of $[n]$ is given by $(n-i)\left(\binom{i}{2}-1\right)\frac{(n-1)!}{i}$ for $i \in [3,n-1]$.  To see this, first note that there are $n-i$ choices for the letter $j$ to play the role of the $3$, given $i$, within an occurrence of $31\dash2$. Let $\sigma \in \mathcal{S}_n$. Note that for each $r$, $1 \leq r \leq i-3$, there are $\frac{(n-1)!}{i}r$ permutations which have an $(i,j)$-occurrence of $31\dash2$ in which the letter $i$ comes somewhere between the $(r+1)$-st and $(r+2)$-nd members of $[i-1]$ from the left within $\text{Flatten}(\sigma)$, and $2\frac{(n-1)!}{i}(i-2)$ permutations in which $i$ occurs to the right of all the members of $[i-1]$ within $\text{Flatten}(\sigma)$.  Observe that in the latter case, the letter $i$ would either occur within a cycle whose first letter is a member of $[i-1]$ or as the first let
 ter of a cycle. Note that in all cases, the possible positions for $j$ are determined by the value of $r$ and is independent of the value of $i$.  Summing over $r$, the total number of $(i,j)$-occurrences of $31\dash2$ is thus given by
$$(1+2+\cdots+(i-3)+2(i-2))\frac{(n-1)!}{i}=\left(\binom{i}{2}-1\right)\frac{(n-1)!}{i}.$$

Summing over $i$, and simplifying, then implies
\begin{align*}
\text{tot}(31\dash2)&=\sum_{i=3}^{n-1}(n-i)\left(\binom{i}{2}-1\right)\frac{(n-1)!}{i}\\
&=\frac{n!}{2}\sum_{i=2}^{n-1}\left(i-1-\frac{2}{i}\right)-(n-1)!\left(\binom{n}{3}-(n-2)\right)\\
&=\frac{n!}{12n}(n^3-3n^2+26n-12)-n!H_n.
\end{align*}
Dividing by $n!$ yields the average value formula found in Corollary \ref{cor:31-2:average}. \hfill \qed\\

\noindent\textbf{Combinatorial proofs of Corollaries \ref{cor:12-3:average} and \ref{cor:21-3:average} (average).}\\

We first treat the case $21\dash3$.  To handle this case, we will simultaneously consider occurrences of the  pattern $3\dash21$.  If $i \in [3,n-1]$, first note that there are $(i-2)(n-1)!$ permutations $\sigma$ of size $n$ in which the letter $i$ directly precedes a member of $[i-1]$ in $\text{Flatten}(\sigma)$.  To show this, first insert $i$ directly before some member of $[2,i-1]$ within a permutation of $[i-1]$ expressed in standard cycle form.  Upon treating $i$ and the letter directly thereafter as a single object, we see that there are $\prod_{s=i+1}^n(s-1)$ choices for the positions of the elements of $[i+1,n]$ and thus the total number of such permutations is $(i-2)(i-1)!\prod_{s=i+1}^n(s-1)=(i-2)(n-1)!$, as claimed.  Within each of these permutations $\sigma$, every letter of $[i+1,n]$ contributes either an $i$-occurrence of $3\dash21$ or $21\dash3$ depending on whether the letter goes somewhere before or somewhere after $i$ within $\text{Flatten}(\sigma)$.  This 
 implies
$$\text{tot}(i\dash\text{occurrences~of } 21\dash3)+\text{tot}(i\dash\text{occurrences~of } 3\dash21)=(n-i)(i-2)(n-1)!, \qquad 3 \leq i \leq n-1,$$
and summing this over $i$ gives
\begin{equation}\label{cpr1}
\text{tot}(21\dash3)+\text{tot}(3\dash21)=(n-1)!\sum_{i=3}^{n-1}(n-i)(i-2).
\end{equation}

We now count the total number of occurrences of $3\dash21$ within permutations of size $n$, which is apparently easier.  We first count the number of permutations having an $(i,j)$-occurrence of $3\dash21$, where $3 \leq i < j \leq n$ are given.  To do so, we first create permutations of the set $[i]\cup\{j\}$ by writing some permutation of $[i-1]$ in standard cycle form and then deciding on the positions of the letters $i$ and $j$.  Either $i$ and $j$ can directly precede different members of $[2,i-1]$ or can precede the same member (in which case $i$ would come before $j$), whence there are $\binom{i-2}{2}+(i-2)=\binom{i-1}{2}$ choices regarding the placement of $i$ and $j$.  Upon treating $i$ and the letter directly thereafter as a single object and adding the remaining members $r$ of $[i+1,n]-\{j\}$, we see that the number of permutations of length $n$ having an $(i,j)$-occurrence of $3\dash21$ is
$$\binom{i-1}{2}(i-1)!\prod_{r=i+1}^{j-1}r \prod_{r=j+1}^n(r-1)=\binom{i-1}{2}\frac{(n-1)!}{i}.$$

Since there are $n-i$ choices for $j$, given $i$, the total number of $i$-occurrence of $3\dash21$ is then given by $(n-i)\binom{i-1}{2}\frac{(n-1)!}{i}.$  Summing over $3 \leq i \leq n-1$ gives
\begin{equation}\label{cpr2}
\text{tot}(3\dash21)=(n-1)!\sum_{i=3}^{n-1}\frac{n-i}{i}\binom{i-1}{2}.
\end{equation}
Subtracting \eqref{cpr2} from \eqref{cpr1} yields
\begin{align*}
\text{tot}(21\dash3)&=(n-1)!\sum_{i=3}^{n-1}(n-i)\left(i-2-\frac{1}{i}\binom{i-1}{2}\right)\\
&=\frac{n!}{2}\sum_{i=2}^{n-1}\left(i-1-\frac{2}{i}\right)-\frac{(n-1)!}{2}\sum_{i=2}^{n-1}(i^2-i-2)\\
&=\frac{n!}{12n}(n^3-3n^2+26n-12)-n!H_n,
\end{align*}
which completes the proof in the case $21\dash3$.

A similar proof may be given for the case $12\dash3$.  In fact, there are the comparable formulas
$$\text{tot}(12\dash3)+\text{tot}(3\dash12)=(n-1)!\sum_{i=2}^{n-1}(n-i)i$$
and
$$\text{tot}(3\dash12)=(n-1)!\sum_{i=2}^{n-1}\frac{n-i}{i}\left(\binom{i}{2}-1\right).$$
Subtracting, simplifying, and dividing by $n!$ gives the average value formula found in Corollary \ref{cor:12-3:average}.  \hfill \qed

\end{document}